\documentclass[12pt, a4paper]{article}
\usepackage{amsfonts}
\usepackage{bbm}
\usepackage{mathrsfs}
\usepackage{latexsym}
\usepackage{amssymb}
\usepackage{amsmath}
\usepackage{color,xcolor}
\usepackage{mathtools}
\usepackage{cite}
\usepackage{amsthm}
\usepackage{amsfonts}
\usepackage{bbm}
\usepackage{mathrsfs}
\usepackage{latexsym}
\usepackage{graphicx}
\usepackage{amssymb}
\usepackage{amsmath}
\usepackage{color,xcolor}
\usepackage{mathtools}
\usepackage{cite}
\usepackage{amsthm}
\usepackage{mathrsfs}
\usepackage{amsmath}
\usepackage{amsfonts,amsthm,amssymb,cite}
\usepackage{graphics,graphicx}
\usepackage{epstopdf}
\usepackage{color}
\oddsidemargin\topmargin=-1.5cm

\usepackage{amssymb,latexsym,bm}
\usepackage{}
\usepackage{color}
\usepackage{amssymb}
\usepackage{amsfonts}
\usepackage{bbm}
\usepackage{mathrsfs}
\usepackage{mathrsfs}
\usepackage{longtable}
\usepackage{geometry}
\usepackage{amsmath,epsfig,cite}
\usepackage{amsthm}
\usepackage{amsmath}
\usepackage{latexsym}
\usepackage{amsmath,bm}
\usepackage{graphics}
\usepackage{multirow}
\usepackage{graphics,epsfig,psfrag}
\usepackage{epstopdf}
\usepackage{caption}
\usepackage{float}
\usepackage{amsfonts,mathrsfs,latexsym,amsmath,amssymb,amsthm}
 \usepackage{cite}
\usepackage{graphicx}
\usepackage{subfigure}

\newtheorem{theorem}{Theorem}[section]
\newtheorem{lemma}[theorem]{Lemma}
\newtheorem{corollary}[theorem]{Corollary}

\newtheorem{conjecture}[theorem]{Conjecture}
\newtheorem{remark}[theorem]{Remark}

\def\det{{\rm det}}

\title{{\Large \bf The Laplacian spectral ratio of connected graphs
\thanks{ Supported by the National Natural Science Foundation of China (Nos. 12261074 and 12071411). }~}}
\author{Zhen Lin$^{1, 2}$\thanks{Corresponding author. E-mail addresses: lnlinzhen@163.com (Z. Lin).}, Jiajia Wang$^{1}$, Min Cai$^{1}$ \\
{\footnotesize $^1$School of Mathematics and Statistics,} \\{ \footnotesize The State Key Laboratory of Tibetan Intelligent Information
Processing and Application,} \\ {\footnotesize
Qinghai Normal University,}\\
 {\footnotesize  Xining, 810008, Qinghai, China}\\
 \footnotesize  $^2$Academy of Plateau Science and Sustainability,
\\  \footnotesize  People's Government of Qinghai Province and Beijing Normal University,\\
\footnotesize Xining, 810008, Qinghai, China
 }

\date{}
\begin{document}
\openup 1.0\jot
\date{}\maketitle
\begin{abstract}
Let $G$ be a simple connected undirected graph. The Laplacian spectral ratio of $G$, denoted by $R_L(G)$, is defined as the quotient between the largest and second smallest Laplacian eigenvalues of $G$, which is closely related to the structural parameters of a graph (or network), such as diameter, $t$-tough, perfect matching, average density of cuts, and synchronizability, etc. In this paper, we obtain some bounds of the Laplacian spectral ratio, which improves the known results. In addition, we give counter-examples on the upper bound of the Laplacian spectral ratio conjecture of trees, and propose a new conjecture.

\bigskip



\end{abstract}
\baselineskip 20pt

\section{\large Introduction}

Let $G$ be a simple connected undirected graph with the vertex set $V(G)$ and edge set $E(G)$. Denote $\overline{G}$ by the complement of $G$. For $v_i\in V(G)$, $N(v{_i})$ denotes the neighborhood of $v_i$
in $G$ and $d(v_i)=|N(v{_i})|$ denotes the degree of vertex $v_i$ in $G$. The maximum degree, the minimum degree, and the diameter of $G$ are denoted by $\Delta$, $\delta$ and $\mathcal{D}$, respectively. The first Zagreb index \cite{GT}, denoted by $Z_1$, is equal to the sum of squares of the degrees of the vertices in $G$, which is a very important topological index in mathematical chemistry. Denote by $P_n$ and $K_{1,\,n-1}$ the path and the star with $n$ vertices.

The Laplacian matrix of $G$ with $n$ vertices, denoted by $L(G)$, is given by $L(G)=D(G)-A(G)$, where $A(G)$ is the adjacency matrix and $D(G)$ is the diagonal matrix of its vertex degrees. And the largest and second smallest Laplacian eigenvalues of $G$ are called the Laplacian spectral radius and algebraic connectivity of $G$, denoted by $\mu_{1}(G)$ and $\mu_{n-1}(G)$.
The Laplacian spectral ratio of a connected graph $G$ with $n$ vertices is defined as
$$R_L(G)=\frac{\mu_1(G)}{\mu_{n-1}(G)}.$$

In 1995, Haemers \cite{H} gave the relationships between diameter, average density of cuts and Laplacian spectral ratio as follows:
$$\mathcal{D}<1+\frac{\log2(n-1)}{\log(\sqrt{R_L(G)}+1)-\log(\sqrt{R_L(G)}-1)},$$
$$\frac{|X||Y|}{(n-|X|)(n-|Y|)}\leq \left(\frac{R_L(G)-1}{R_L(G)+1}\right)^2,$$
where $X$ and $Y$ are disjoint sets of vertices of $G$, that is no edge between $X$ and $Y$.
In 2006, Goldberg \cite{G} obtained a lower bound on Laplacian spectral ratio based on the maximum degree and the minimum degree:
$$R_L(G)\geq \frac{\Delta+1}{\delta},$$
when $G$ is not a complete graph. Brouwer and Haemers \cite{BH} proved that a connected graph of even order $n$ satisfies $R_L(G)\leq 2$, then $G$ has a perfect matching. Liu and Chen \cite{LC} showed that if $R_L(G)\leq \frac{3}{2}$, then $G$ is $2$-tough. Let $G$ be a connected graph with $n$ vertices and $m$ edges. By the Kantorovich inequality \cite{K}, we have
$$Kf(G)\leq \frac{n(n-1)^2}{8m}\left(R_L(G)+\frac{1}{R_L(G)}+2\right),$$
where $Kf(G)$ is called the Kirchhoff index. Recently, Lin \cite{L1} established mathematical relations between the biharmonic index and the Laplacian spectral ratio.

On the other hand, Barahona et al. \cite{BP} showed that a network $G$ exhibits better synchronizability when the ratio $R_L(G)$ is as small as possible. Since the synchronizability depends on $R_L(G)$,  Arenas et al. \cite{ADKMZ} believed that a sound analysis must attack the raw problem of the spectral properties of networks from a mathematical point of view, especially focus on the bounds of $R_L(G)$, given that the simulation experiments are far from being conclusive.

For trees, You and Liu \cite{YL} proposed the following conjecture, and gave the conditions for the conjecture to hold in terms of diameter and maximum degree.

\begin{conjecture}{\bf (\cite{YL})}\label{con1,1} 
Let $T$ be a tree with  $n\geq 3$ vertices. Then
$$R_L(K_{1,\,n-1})\leq R_L(T)\leq R_L(P_n),$$
and the left (right) equality holds if and only if $T= K_{1,\,n-1}$ ($T= P_n$).
\end{conjecture}

\begin{theorem}{\bf (\cite{YL})}\label{th1,1} 
Let $T\neq K_{1, \, n-1}$ be a tree with $n\geq 10$ vertices. If maximum degree $\Delta(T)\geq \lceil \frac{n}{2}\rceil-1$ or diameter $\mathcal{D}(T)\geq \lceil \frac{n}{2}\rceil-1$, then $R_L(T)>R_L(K_{1, \, n-1})$.
\end{theorem}

In this paper, we obtain some bounds on Laplacian spectral ratio by rank one perturbation matrix and quotient matrix, and compare with the known results. Moreover, we improve the conditions of the Theorem \ref{th1,1}. For
the upper bound of the Conjecture \ref{con1,1}, we find counter-examples using the numerical results of all trees with nine vertices. Finally, we give the Laplacian spectral ratio of some special trees, and propose a new conjecture on the upper bound of the Laplacian spectral ratio of trees.

\section{\large  Preliminaries}

\begin{lemma}{\bf (\cite{WS})}\label{le2,1} 
Let $M_{n\times n}$ be a positive definite Hermitian matrix, and let
$$r=\frac{tr(M)}{n}, \qquad s^2=\frac{tr(M^2)}{n}-r^2.$$
If $n$ is even, then
$$\frac{\lambda_{\max}(M)}{\lambda_{\min}(M)}\geq 1+\frac{2s}{r-s/\sqrt{n-1}}.$$
If $n$ is odd, then
$$\frac{\lambda_{\max}(M)}{\lambda_{\min}(M)}\geq 1+\frac{2sn/\sqrt{n^2-1}}{r-s/\sqrt{n-1}}.$$
\end{lemma}

Let $M$ be a real symmetric partitioned matrix of order $n$ described in the following block form
\[  \begin{pmatrix}
M_{11} &  \cdots & M_{1t} \\
\vdots & \ddots & \vdots \\
M_{t1} &  \cdots & M_{tt}
\end{pmatrix}, \]
where the diagonal blocks $M_{ii}$ are $n_i\times n_i$ matrices for any $i \in \{1,2,\ldots,t\}$ and $n=n_1+\cdots+n_t$.
For any $i, j \in \{1,2,\ldots,t\}$, $b_{ij}$ is the average row
sum of $M_{ij}$ , i.e. $b_{ij}$ is the sum of all entries in $M_{ij}$ divided by the number of
rows. Then $\mathcal{B}(M)=(b_{ij})$ (simply by $\mathcal{B}$) is called the quotient matrix of $M$.

\begin{lemma}{\bf (\cite{H})}\label{le2,2} 
Let $M$ be a symmetric partitioned matrix of order $n$ with eigenvalues $\xi_1\geq \xi_2 \geq \cdots  \geq \xi_n$, and let $\mathcal{B}$
its quotient matrix with eigenvalues $\eta_1\geq \eta_2 \geq \cdots  \geq \eta_m$ and $n>m$. Then $\xi_i \geq \eta_i \geq \xi_{n-m+i}$
for $i=1, 2, \ldots, m$.
\end{lemma}

\begin{lemma}{\bf (\cite{GM, K1})}\label{le2,3} 
Let $G$ be a graph with $n$ vertices and at least one edge. Then $\Delta+1\leq \mu_1(G)\leq n$. The left equality for
connected graph holds if and only if $\Delta=n-1$, and the right equality holds if and only if $\overline{G}$ is disconnected.
\end{lemma}

\begin{lemma}{\bf (\cite{F})}\label{le2,4} 
Let $G$ be not a complete graph with $n$ vertices. Then $\mu_{n-1}(G)\leq \delta$.
\end{lemma}

\begin{lemma}{\bf (\cite{Z})}\label{le2,5} 
Let $G$ be a bipartite graph with $m\geq 1$ edges. Then $\mu_1(G)\geq \frac{Z_1}{m}$.
\end{lemma}

\begin{lemma}{\bf (\cite{WS})}\label{le2,6} 
Let $M_{n\times n}$ be a positive definite Hermitian matrix, and let
$$r=\frac{tr(M)}{n}, \qquad s^2=\frac{tr(M^2)}{n}-r^2.$$
Then
$$\frac{\lambda_{\max}(M)}{\lambda_{\min}(M)}\leq 1+\frac{s\sqrt{2n}(r+s/\sqrt{n-1})^{n-1}}{\det M}.$$
\end{lemma}

\begin{lemma}{\bf (\cite{M1})}\label{le2,7} 
Let $\overline{G}$ be the complement of the graph $G$ with $n$ vertices. The Laplacian eigenvalues of $\overline{G}$ are
$$n-\mu_{n-1}(G)\geq n-\mu_{n-2}(G)\geq \cdots \geq n-\mu_1(G)\geq 0.$$
\end{lemma}

\begin{lemma}{\bf (\cite{GMS})}\label{le2,8} 
If $T$ is a tree with diameter $\mathcal{D}$, then $\mu_{n-1}(T)\leq 2\left(1-\cos\frac{\pi}{\mathcal{D}+1}\right)$.
\end{lemma}

\section{\large Bounds on Laplacian spectral ratio}

\begin{theorem}\label{th3,1} 
Let $G$ be a connected graph with $n$ vertices, $m$ edges and the first Zagreb index $Z_1$. If $n$ is even, then
$$R_L(G)\geq 1+\frac{2\sqrt{(n-1)[n^2(n-1)\alpha^2-4mn\alpha+nZ_1+2mn-4m^2]}}{(2m+n\alpha)\sqrt{n-1}-\sqrt{n^2(n-1)\alpha^2-4mn\alpha+nZ_1+2mn-4m^2}},$$
where $\alpha$ is a real number and $\mu_{n-1}\leq n\alpha \leq \mu_1$.
If $n$ is odd, then
$$R_L(G)\geq 1+\frac{2n\sqrt{n^2(n-1)\alpha^2-4mn\alpha+nZ_1+2mn-4m^2}}{(2m+n\alpha)\sqrt{n^2-1}-\sqrt{(n+1)[n^2(n-1)\alpha^2-4mn\alpha+nZ_1+2mn-4m^2]}},$$
where $\alpha$ is a real number and $\mu_{n-1}\leq n\alpha \leq \mu_1$.
\end{theorem}

\begin{proof} We consider the matrix $M=L(G)+\alpha J$, where $J$ is the all-ones matrix. Then
$$m_{ij}=
\begin{dcases}
d_i+\alpha, & \text{if}\,\, i=j ;\\
-1+\alpha, & \text{if}\,\, v_iv_j\in E(G);\\
\alpha, & \text{if}\,\, v_iv_j\notin E(G).
\end{dcases}
$$
Thus we have
\begin{eqnarray*}
tr(M) & = & \sum_{i=1}^{n}(d_i+\alpha)=2m+n\alpha,\\
tr(M^2) & = & \sum_{i=1}^{n}[(d_i+\alpha)^2+d_i(-1+\alpha)^2+(n-1-d_i)\alpha^2]=Z_1+2m+n^2\alpha^2.
\end{eqnarray*}
Moreover, the eigenvalues of $M$ are $\mu_1, \mu_2, \ldots, \mu_{n-1}, \alpha n$. Since $\mu_{n-1}\leq n\alpha \leq \mu_1$, we have that $M$ is a positive definite Hermitian matrix, and $R_L(G)=\frac{\lambda_{\max}(M)}{\lambda_{\min}(M)}$.

If $n$ is even, by Lemma \ref{le2,1}, we have
$$\frac{\lambda_{\max}(M)}{\lambda_{\min}(M)}\geq 1+\frac{2\sqrt{(n-1)[n^2(n-1)\alpha^2-4mn\alpha+nZ_1+2mn-4m^2]}}{(2m+n\alpha)\sqrt{n-1}-\sqrt{n^2(n-1)\alpha^2-4mn\alpha+nZ_1+2mn-4m^2}},$$
that is,
$$R_L(G)\geq 1+\frac{2\sqrt{(n-1)[n^2(n-1)\alpha^2-4mn\alpha+nZ_1+2mn-4m^2]}}{(2m+n\alpha)\sqrt{n-1}-\sqrt{n^2(n-1)\alpha^2-4mn\alpha+nZ_1+2mn-4m^2}}.$$

If $n$ is odd, by Lemma \ref{le2,1}, we have
$$\frac{\lambda_{\max}(M)}{\lambda_{\min}(M)}\geq 1+\frac{2n\sqrt{n^2(n-1)\alpha^2-4mn\alpha+nZ_1+2mn-4m^2}}{(2m+n\alpha)\sqrt{n^2-1}-\sqrt{(n+1)[n^2(n-1)\alpha^2-4mn\alpha+nZ_1+2mn-4m^2]}},$$
that is,
$$R_L(G)\geq 1+\frac{2n\sqrt{n^2(n-1)\alpha^2-4mn\alpha+nZ_1+2mn-4m^2}}{(2m+n\alpha)\sqrt{n^2-1}-\sqrt{(n+1)[n^2(n-1)\alpha^2-4mn\alpha+nZ_1+2mn-4m^2]}}.$$
This completes the proof. $\Box$
\end{proof}

\begin{corollary}\label{cor3,1} 
Let $G$ be a connected bipartite graph with $n$ vertices and $m$ edges. If $n$ is even, then
$$R_L(G)\geq 1+\frac{2\sqrt{(n-1)[(n-1)Z_1^2-4m^2Z_1+nm^2Z_1+2nm^3-4m^4]}}{(2m^2+Z_1)\sqrt{n-1}-\sqrt{(n-1)Z_1^2-4m^2Z_1+nm^2Z_1+2nm^3-4m^4}}.$$
If $n$ is odd, then
$$R_L(G)\geq 1+\frac{2n\sqrt{(n-1)Z_1^2-4m^2Z_1+nm^2Z_1+2nm^3-4m^4}}{(2m^2+Z_1)\sqrt{n^2-1}-\sqrt{(n+1)[(n-1)Z_1^2-4m^2Z_1+nm^2Z_1+2nm^3-4m^4]}}.$$
\end{corollary}

\begin{proof}
By Lemma \ref{le2,5}, we take $\alpha=\frac{Z_1}{mn}$ in Theorem \ref{th3,1}, the result follows.  $\Box$
\end{proof}

\begin{theorem}\label{th3,2} 
Let $G$ be a connected $k$-regular triangle-free graph with $n$ vertices. Then
$$R_L(G)\geq \frac{2kn-k^2-3k+\sqrt{4kn^2-4k(3k+1)n+k^4+6k^3+9k^2}}{2kn-k^2-3k-\sqrt{4kn^2-4k(3k+1)n+k^4+6k^3+9k^2}}.\eqno{(3.1)}$$
The equality holds if $G$ is a Petersen graph.
\end{theorem}

\begin{proof} Let $\mathcal{B}(G)$ be the quotient matrix of $L(G)$ corresponding to the partition $V(G)=\{v_1\}\cup N(v_1) \cup V(G)\setminus (\{v_1\}\cup N(v_1))$ of $G$. Then
\begin{equation*}
\mathcal{B}(G)=\left(
  \begin{array}{ccccccccccc}
  k & -k & 0\\
 -1 & k & -(k-1)\\
  0 & \frac{-k(k-1)}{n-k-1} & \frac{k(k-1)}{n-k-1}\\
\end{array}
\right). \tag{3.2}
\end{equation*}
By direct computation the characteristic polynomial of (3.2) is
$$\det(xI_n-\mathcal{B}(G))=\frac{x}{n-k-1}[(n-k-1)x^2-(2kn-k^2-3k)x+k^2n-kn].$$
Thus
$$\frac{\eta_1(G)}{\eta_{2}(G)}=\frac{2kn-k^2-3k+\sqrt{4kn^2-4k(3k+1)n+k^4+6k^3+9k^2}}{2kn-k^2-3k-\sqrt{4kn^2-4k(3k+1)n+k^4+6k^3+9k^2}}.$$
By Lemma \ref{le2,2}, we have
$$R_L(G)\geq \frac{\eta_1(G)}{\eta_{2}(G)}=\frac{2kn-k^2-3k+\sqrt{4kn^2-4k(3k+1)n+k^4+6k^3+9k^2}}{2kn-k^2-3k-\sqrt{4kn^2-4k(3k+1)n+k^4+6k^3+9k^2}}.$$
This completes the proof. $\Box$
\end{proof}

\begin{remark} In 2012, You and Liu \cite{YL} obtained the lower bound of connected $k$-regular graph as follows:
$$R_L(G)\geq \left(\sqrt{\frac{(n-1)(k+1)}{nk}}+\sqrt{\frac{(n-1)(k+1)}{nk}-1}\right)^2. \eqno{(3.3)}$$
If $G$ is a cycle $C_{10}$, applying (3.1) and (3.3), we have $R_L(G)\geq 4.1899$ and $R_L(G)\geq 3.0748$, respectively. In fact, $R_L(C_{10})\approx10.4721$.
This example shows that our result is better than known result for some special graphs.
\end{remark}

\begin{theorem}\label{th3,3} 
Let $G$ be a connected graph with $n$ vertices, $m$ edges, $\tau$ spanning trees and the first Zagreb index $Z_1$. Then
$$R_L(G)\leq 1+\frac{\sqrt{2\Omega}\left(2m+n\alpha+\sqrt{\frac{\Omega}{n-1}}\right)^{n-1}}{n^{n+3/2}\alpha \tau},$$
where $\mu_{n-1}\leq n\alpha \leq \mu_1$ and $\Omega=n^2(n-1)\alpha^2-4mn\alpha+nZ_1+2mn-4m^2$.
\end{theorem}

\begin{proof}
Let $M=L(G)+\alpha J$, where $J$ is the all-ones matrix. By the matrix-tree theorem, we have $\det M=n\alpha \prod_{i=1}^{n-1}\mu_i=n^2\alpha\tau$. By Lemma \ref{le2,6} and the proof of Theorem \ref{th3,1}, we have the proof. $\Box$
\end{proof}

\begin{theorem}\label{th3,4} 
Let $G$ be a connected graph with $n$ vertices, maximum degree $\Delta$ and minimum degree $\delta$. If $\overline{G}$ is a connected graph, then
$$R_L(G)+R_L(\overline{G})\geq \frac{\Delta+1}{\delta}+\frac{n-\delta}{n-\Delta-1}.$$
\end{theorem}

\begin{proof} By Lemma \ref{le2,7}, we have
$$R_L(G)+R_L(\overline{G})=\frac{\mu_1(G)}{\mu_{n-1}(G)}+\frac{n-\mu_{n-1}(G)}{n-\mu_1(G)}.$$
Let $f(x, y): = \frac{x}{y}+\frac{n-y}{n-x}$ with $0<x<n$ and $0<y<n$. It is easy to prove that the function $f(x, y)$ is
increasing in $x$ and decreasing in $y$. By Lemmas \ref{le2,3} and \ref{le2,4}, we have
$$R_L(G)+R_L(\overline{G})\geq f(\Delta+1, \delta)=\frac{\Delta+1}{\delta}+\frac{n-\delta}{n-\Delta-1}.$$
This completes the proof. $\Box$
\end{proof}

\section{\large  Laplacian spectral ratio of trees}

\begin{theorem}\label{th4,1} 
Let $T\neq K_{1,\, n-1}$ be a tree with $n$ vertices and diameter $\mathcal{D}$. If $\mathcal{D}\geq \pi \sqrt{\frac{n}{4}+\frac{3}{16n-24}+\frac{1}{8}}-1$, then
$$R_L(T)> R_L(K_{1,\, n-1}).$$
\end{theorem}

\begin{proof} From \cite{GD}, we have $Z_1\geq 4n-6$ for any tree. By Lemma \ref{le2,5}, we have
$$\mu_1(T)\geq \frac{Z_1}{n-1}\geq \frac{4n-6}{n-1}.$$
If $\mathcal{D} \geq \pi \sqrt{\frac{n}{4}+\frac{3}{16n-24}+\frac{1}{8}}-1$, by Lemma \ref{le2,8}, we have
$$R_L(T)=\frac{\mu_1(T)}{\mu_{n-1}(T)}\geq \frac{4n-6}{2(n-1)\left(1-\cos\frac{\pi}{\mathcal{D}+1}\right)}>\frac{4n-6}{(n-1)\left(\frac{\pi}{\mathcal{D}+1}\right)^2}\geq n=R_L(K_{1,\, n-1}).$$
This completes the proof. $\Box$
\end{proof}

\begin{remark}
The condition in Theorem \ref{th4,1} is always better than that in Theorem \ref{th1,1}.
\end{remark}

\begin{theorem}\label{th4,2} 
Let $T\neq K_{1,\, n-1}$ be a tree with $n\geq 6$ vertices and with exactly $k$ vertices of maximum degree $\Delta$. If $\Delta \geq \sqrt{\frac{n(n-3)}{2k}+1}$, then
$$R_L(T)> R_L(K_{1,\, n-1}).$$
\end{theorem}

\begin{proof} Since $Z_1\geq k\Delta^2+n-k$, by Lemma \ref{le2,5}, we have
$$\mu_1(T)\geq \frac{Z_1}{n-1}\geq \frac{k\Delta^2+n-k}{n-1}.$$
From \cite{GMS}, we have that $\mu_{n-1}(T)<\frac{1}{2}$ when $T\neq K_{1,\, n-1}$ is a tree on $n\geq 6$ vertices.
Thus
$$R_L(T)=\frac{\mu_1(T)}{\mu_{n-1}(T)}> \frac{2(k\Delta^2+n-k)}{n-1}\geq n=R_L(K_{1,\, n-1})$$
for $\Delta \geq \sqrt{\frac{n(n-3)}{2k}+1}$.
This completes the proof. $\Box$
\end{proof}

\begin{remark}
The condition in Theorem \ref{th4,2} is always better than that in Theorem \ref{th1,1}.
\end{remark}

Let $T$ be a tree, and let $d(u, v)$ be the distance between vertices $u$ and $v$ in $T$. Define
$$\mathfrak{D}(T)=\{v\in V(T): \text{there}\,\, \text{is}\,\, \text{some} \,\, x\in V(G)\,\, \text{with}\,\, d(x,v)\geq 3\}.$$
In \cite{BEHK}, Barrett et al. showed that
$$\mu_{n-1}(T)\geq \frac{n-s+1-\sqrt{(n-s+1)^2-4(n-2s)}}{2},$$
where $s=|\mathfrak{D}(T)|/2$ and $\mathfrak{D}(T)$ consists of all the vertices of eccentricity at least three.

\begin{theorem}\label{th4,3} 
Let $T$ be a tree with $n\geq 6$ vertices. If $|\mathfrak{D}(T)|\leq n-1$, then
$$R_L(T)< R_L(P_n).$$
\end{theorem}

\begin{proof} Let $g(x): =\tan x-\frac{11}{10}x$. And $g'(x)=\frac{1}{\cos^2x}-\frac{11}{10}<0$ for $0<x<\frac{\pi}{11}$. Then $g(x)$ is decreasing function in the interval $(0, \frac{\pi}{11})$. Thus we have
$g(\frac{\pi}{2n})=\tan \frac{\pi}{2n}-\frac{11}{10}\cdot \frac{\pi}{2n}< g(0)=0$ for $n\geq 6$. It is easy to see that $\mu_1(P_n)=2+2\cos\frac{\pi}{n}$ and $\mu_{n-1}(P_n)=2-2\cos\frac{\pi}{n}$. Hence
$$R_L(P_n)=\frac{1+\cos\frac{\pi}{n}}{1-\cos\frac{\pi}{n}}=\frac{1}{\tan^2\frac{\pi}{2n}}>\frac{400n^2}{121\pi^2}.$$
If $s\leq \frac{n}{2}-\frac{1}{2}$, that is $|\mathfrak{D}(T)|\leq n-1$, by Lemma \ref{le2,3}, we have
$$R_L(T)=\frac{\mu_1(T)}{\mu_{n-1}(T)}\leq \frac{2n}{n-s+1-\sqrt{(n-s+1)^2-4(n-2s)}}\leq \frac{400n^2}{121\pi^2}<R_L(P_n)$$
for $n\geq 3$. This completes the proof. $\Box$
\end{proof}

\begin{remark} Let $T^{*}$ be the tree graph with $n$ vertices obtained from a star $K_{1,\, \frac{n}{2}}$ by joining $\frac{n}{2}$ pendant edges of $K_{1,\, \frac{n}{2}}$ to the $\frac{n}{2}-1$
isolated vertices by $\frac{n}{2}-1$ edges. It is not difficult to see that $|\mathfrak{D}(T^{*})|= n-1$. By the direct computation, we have the Laplacian characteristic polynomial of $T^{*}$ is:
$$\Phi(T^{*}, x)=x(x-2)(x^2-3x+1)^{\frac{n}{2}-2}\left[x^2-(\frac{n}{2}+2)x+\frac{n}{2}\right].$$
Thus we have
$$n=R_L(K_{1,\,n-1})<R_L(T^{*})=\frac{n+4+\sqrt{n^2+16}}{6-2\sqrt{5}}<\frac{400n^2}{121\pi^2}<R_L(P_n)$$
for $n\geq 6$.
\end{remark}

\begin{corollary}\label{cor4,1} 
Let $T$ be a tree with $n\geq 6$ vertices and diameter $\mathcal{D}$. If $\mathcal{D}\leq 4$, then
$$R_L(T)< R_L(P_n).$$
\end{corollary}

\begin{proof}
If $\mathcal{D}\leq 4$, then $|\mathfrak{D}(T)|\leq n-1$. By Theorem \ref{th4,3},  we have the proof.  $\Box$
\end{proof}

\begin{theorem}\label{th4,4} 
Let $T$ be a tree with $n\geq 6$ vertices, maximum degree $\Delta$ and diameter $\mathcal{D}$. If $\Delta+2\sqrt{\Delta+1}\leq \frac{1600n}{121\mathcal{D}\pi^2}$, then
$$R_L(T)< R_L(P_n).$$
\end{theorem}

\begin{proof} Form \cite{M} and \cite{RR}, we have $\mu_1(T)<\Delta+2\sqrt{\Delta+1}$ and $\mu_{n-1}(T)\geq \frac{4}{n \mathcal{D}}$. Thus
$$R_L(T)<\frac{n\mathcal{D}(\Delta+2\sqrt{\Delta+1})}{4}\leq \frac{400n^2}{121\pi^2},$$
that is,
$$\Delta+2\sqrt{\Delta+1}\leq \frac{1600n}{121\mathcal{D}\pi^2}.$$
This completes the proof. $\Box$
\end{proof}

\section{\large  Laplacian spectral ratio of special trees}

Let $T_{\Delta,\, \mathcal{D}}$ be a caterpillar tree with $n$ vertices, maximum degree $\Delta$ and diameter $\mathcal{D}\geq3$, and $(\Delta-1)(\mathcal{D}-1)=n$, shown in Figure 1.

\begin{figure}[!hbpt]
\begin{center}
\includegraphics[scale=0.88]{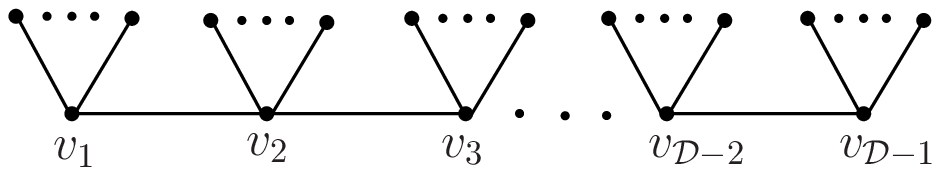}\\
Figure 1. ~ $T_{\Delta,\, \mathcal{D}}$, $d(v_1)=d(v_{\mathcal{D}-1})=\Delta-1$, $d(v_2)=\cdots=d(v_{\mathcal{D}-2})=\Delta$.
\end{center}\label{fig02}
\end{figure}

\begin{theorem}\label{th5,1} 
Let $T_{\Delta,\, \mathcal{D}}$ be a tree with $n\geq 5$ vertices, maximum degree $\Delta$ and diameter $\mathcal{D}\geq3$. Then
$$R_L(T_{\Delta,\,\mathcal{D}})=\frac{\Delta+1+2\cos\frac{\pi}{\mathcal{D}-1}+\sqrt{\left(\Delta-1+2\cos\frac{\pi}{\mathcal{D}-1}\right)^2+4(\Delta-2)}}
{\Delta+1-2\cos\frac{\pi}{\mathcal{D}-1}-\sqrt{\left(\Delta-1-2\cos\frac{\pi}{\mathcal{D}-1}\right)^2+4(\Delta-2)}}.$$
\end{theorem}

\begin{proof} Let $I$ be an identity matrix of appropriate size. By the definition of $T_{\Delta,\, \mathcal{D}}$, we have
\begin{equation*}
L(T_{\Delta,\, \mathcal{D}})=\left(
  \begin{array}{ccccccccccc}
  B & \beta_1 & \beta_2 & \cdots & \beta_{\mathcal{D}-1}\\
  \beta_1^T & I & O & \cdots & O \\
  \beta_2^{T} & O & I & \cdots & O\\
  \vdots & \vdots & \vdots & \ddots& \vdots \\
  \beta_{\mathcal{D}-1}^T & O & O & \cdots & I\\
\end{array}
\right),
\end{equation*}
where
\begin{equation*}
B=\left(
  \begin{array}{ccccccccccc}
  \Delta-1 & -1 & 0 & 0 & \cdots & 0 & 0 & 0\\
  -1 &  \Delta & -1 & 0 & \cdots & 0 & 0 & 0\\
  0 &  -1 & \Delta & -1 & \cdots & 0 & 0 & 0\\
  0 &  0 & -1 & \Delta & \cdots & 0 & 0 & 0\\
  0 &  0 & 0 & -1 & \ddots & 0 & 0 & 0\\
  \vdots &  \vdots & \vdots & \vdots & \ddots & \ddots & \vdots & \vdots\\
  0 &  0 & 0 & 0 & \cdots & -1 & \Delta & -1\\
  0 &  0 & 0 & 0 & \cdots & 0 & -1 & \Delta-1\\
\end{array}
\right)_{(\mathcal{D}-1)\times (\mathcal{D}-1)},
\end{equation*}
\begin{equation*}
\beta_1=\left(
  \begin{array}{ccccccccccc}
 -1 & -1 & \cdots & -1 \\
 0 & 0 & \cdots & 0 \\
 \vdots & \vdots & \cdots & \vdots \\
  0 & 0 & \cdots & 0 \\
\end{array}
\right)_{(\mathcal{D}-1)\times (\Delta-2)},
\beta_2=\left(
  \begin{array}{ccccccccccc}
 0 & 0 & \cdots & 0 \\
 -1 & -1 & \cdots & -1 \\
 \vdots & \vdots & \cdots & \vdots \\
  0 & 0 & \cdots & 0 \\
\end{array}
\right)_{(\mathcal{D}-1)\times (\Delta-2)},\ldots.
\end{equation*}
Thus we have the Laplacian characteristic polynomial of $T_{\Delta, \mathcal{D}}$ is:
\begin{eqnarray*}
\Phi(T_{\Delta, \mathcal{D}},x) & = & \det(xI-L(T_{\Delta, \mathcal{D}}))\\
& = & \det^{\mathcal{D}-1}(xI-I)\cdot \det\left(\left(x+(\Delta-2)\frac{1}{1-x}\right)I-B\right)\\
& = & (x-1)^{(\mathcal{D}-1)(\Delta-2)}\det\left(\left(x+(\Delta-2)\frac{1}{1-x}\right)I-B\right).
\end{eqnarray*}
Note that the eigenvalues of $B$ are:
$$\lambda_i(B)=\Delta-2\cos\frac{(i-1)\pi}{\mathcal{D}-1}, \quad i=1, 2, \ldots, \mathcal{D}-1.$$
By Lemmas \ref{le2,3} and \ref{le2,4}, $\mu_1(T_{\Delta, \mathcal{D}})$ and $\mu_{n-1}(T_{\Delta, \mathcal{D}})$ are the roots of the equation
$$x+(\Delta-2)\frac{1}{1-x}=\Delta-2\cos\frac{(i-1)\pi}{\mathcal{D}-1}, \quad i=1, 2, \ldots, \mathcal{D}-1.$$
Therefore, we have
\begin{eqnarray*}
\mu_1(T_{\Delta, \mathcal{D}}) & = & \frac{\Delta+1+2\cos\frac{\pi}{\mathcal{D}-1}+\sqrt{\left(\Delta-1+2\cos\frac{\pi}{\mathcal{D}-1}\right)^2+4(\Delta-2)}}{2},\\
\mu_{n-1}(T_{\Delta, \mathcal{D}}) & = & \frac{\Delta+1-2\cos\frac{\pi}{\mathcal{D}-1}-\sqrt{\left(\Delta-1-2\cos\frac{\pi}{\mathcal{D}-1}\right)^2+4(\Delta-2)}}{2}.
\end{eqnarray*}
Further,
$$R_L(T_{\Delta,\, \mathcal{D}})
=\frac{\Delta+1+2\cos\frac{\pi}{\mathcal{D}-1}+\sqrt{\left(\Delta-1+2\cos\frac{\pi}{\mathcal{D}-1}\right)^2+4(\Delta-2)}}{\Delta+1-2\cos\frac{\pi}{\mathcal{D}-1}
-\sqrt{\left(\Delta-1-2\cos\frac{\pi}{\mathcal{D}-1}\right)^2+4(\Delta-2)}}.$$
This completes the proof. $\Box$
\end{proof}

\begin{corollary}\label{cor5,1} 
Let $T_{\Delta,\, \mathcal{D}}$ be a tree with $n\geq 5$ vertices, maximum degree $\Delta$ and diameter $\mathcal{D}\geq3$. Then
$$R_L(K_{1,\,n-1})<R_L(T_{\Delta,\, \mathcal{D}})<R_L(P_n).$$
\end{corollary}

\begin{proof} Let $h(x, y): =\frac{x+1+2y+\sqrt{\left(x-1+2y\right)^2+4(x-2)}}{x+1-2y-\sqrt{\left(x-1-2y\right)^2+4(x-2)}}$. By derivation, we know that $h(x, y)$ is increasing for $3\leq x\leq \frac{n}{2}+1$ and $0\leq y<1$.
Note that $(\Delta-1)(\mathcal{D}-1)=n$. Comparing by calculation, the maximum and minimum values of $h(x,y)$ are respectively obtained at the following places:
$$\left(3, \cos\frac{2\pi}{n}\right), \quad \left(\frac{n+2}{2}, 0\right).$$
Thus we have
$$R_L(T_{\frac{n+2}{2},\, 3})\leq R_L(T_{\Delta,\, \mathcal{D}})\leq R_L(T_{3,\, \frac{n+2}{2}}).$$

Since $R_L(T_{\frac{n+2}{2},\, 3}) =  \frac{n+4+\sqrt{n^2+8n-16}}{n+4-\sqrt{n^2+8n-16}} > n$
for $n\geq 3$, we have $R_L(T_{\frac{n+2}{2},\, 3}) > R_L(K_{1,\,n-1})$.

On the other hand, it suffices to show that
\begin{eqnarray*}
R_L(T_{3,\, \frac{n+2}{2}}) & = & \frac{4+2\cos\frac{2\pi}{n}+\sqrt{\left(2+2\cos\frac{2\pi}{n}\right)^2+4}}{4-2\cos\frac{2\pi}{n}-\sqrt{\left(2-2\cos\frac{2\pi}{n}\right)^2+4}}\\
& = & \frac{1+2\cos^2\frac{\pi}{n}+\sqrt{4\cos^4\frac{\pi}{n}+1}}{3-2\cos^2\frac{\pi}{n}-\sqrt{4\left(1-\cos^2\frac{\pi}{n}\right)^2+1}}\\
& < & \frac{1+\cos\frac{\pi}{n}}{1-\cos\frac{\pi}{n}}\\
& = & R_L(P_n),
\end{eqnarray*}
which is equivalent to,
$$(1-\cos\frac{\pi}{n})\sqrt{4\cos^4\frac{\pi}{n}+1}+(1+\cos\frac{\pi}{n})\sqrt{4\left(1-\cos^2\frac{\pi}{n}\right)^2+1}<2-4\cos^2\frac{\pi}{n}+4\cos\frac{\pi}{n},$$
that is,
$$16\cos\frac{2\pi}{n}\left(1-\cos\frac{\pi}{n}\right)^2\left(8\cos^5\frac{\pi}{n}+10\cos^4\frac{\pi}{n} -4\cos^3\frac{\pi}{n} +\cos^2\frac{\pi}{n} +4\cos\frac{\pi}{n}+1\right)>0,$$
which holds evidently for $n\geq 5$. Thus we have
$R_L(T_{3,\, \frac{n+2}{2}})<R_L(P_n)$.

This completes the proof. $\Box$
\end{proof}

A broom tree $B_{n}^{t}$, shown in Figure 2,  is a tree obtained from the path $P_{n-t-1}$ and the star $K_{1,\,t}$ by joining one pendant vertex of $P_{n-t-1}$ and the center of $K_{1,\,t}$ by an edge.
Now, we give the numerical results of Laplacian spectral ratio of trees with nine vertices, shown in Figure 3. We find that $R_L(B_{9}^{3})>R_L(B_{9}^{2})>R_L(B_{9}^{4})>R_L(P_{9})$.
This shows that the upper bound of Conjecture \ref{con1,1} is not correct. Therefore, we obtain the Laplacian characteristic polynomial of $B_{n}^{t}$, and propose a new conjecture.

\begin{figure}[!hbpt]
\begin{center}
\includegraphics[scale=0.88]{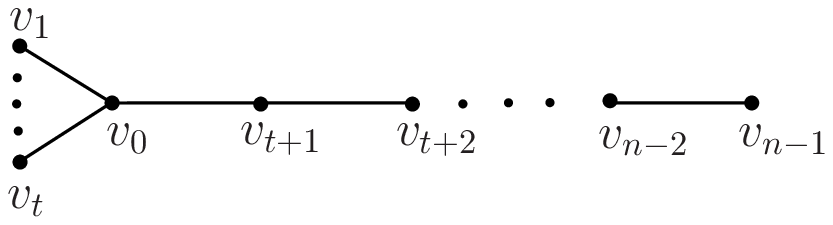}\\
Figure 2. ~ The broom tree $B_{n}^{t}$.
\end{center}\label{fig02}
\end{figure}

\begin{theorem}\label{th5,2} 
Let $B_{n}^{t}$ be a broom tree with $n$ vertices. Then the Laplacian characteristic polynomial of $B_{n}^{t}$ is
\begin{eqnarray*}
\Phi(B_{n}^{t}, \lambda) & = & (\lambda-1)^{t-1}(\lambda^2-(t+2)\lambda+1)\sum_{i=0}^{n-t-1}(-1)^{i}\binom{2n-2t-2-i}{i}\lambda^{n-t-1-i}\\
& & -(\lambda-1)^t\sum_{i=0}^{n-t-2}(-1)^{i}\binom{2n-2t-4-i}{i}\lambda^{n-t-2-i}.
\end{eqnarray*}
\end{theorem}

\begin{proof} Let $X=(x_0, x_1, \ldots, x_{n-1})^T $ be an eigenvector corresponding to any eigenvalue $\lambda$.
By the eigenvalue equation $L(B_{n}^{t})X =\lambda X$, we have $x_1=x_2=\cdots=x_t$,
$$
\begin{cases}
(\lambda-t-1)x_0+tx_1+x_{t+1}=0,\\
x_0+(\lambda-1)x_1=0,\\
x_0+(\lambda-2)x_{t+1}+x_{t+2}=0,\\
x_{t+1}+(\lambda-2)x_{t+2}+x_{t+3}=0,\\
\vdots\\
x_{n-3}+(\lambda-2)x_{n-2}+x_{n-1}=0,\\
x_{n-2}+(\lambda-1)x_{n-1}=0.
\end{cases}$$
Since $X$ is an eigenvector, it follows that
$$\det(\lambda I-L')=
\begin{vmatrix}
\lambda-t-1& t & 1 & 0  & 0 & \cdots & 0 & 0 & 0\\
1 &\lambda-1 & 0 & 0 & 0 & \cdots & 0 & 0 & 0\\
1 & 0 & \lambda-2 & 1 & 0 & \cdots & 0 & 0 & 0  \\
0 & 0 & 1 & \lambda-2 & 1 & \cdots & 0 & 0 & 0  \\
0 & 0 & 0 & 1 & \lambda-2 &  \cdots & 0 & 0 & 0  \\
\vdots & \vdots & \vdots & \vdots &\vdots&  \ddots & \vdots & \vdots & \vdots  \\
0 & 0 & 0 & 0 & 0 &  \cdots & \lambda-2 & 1 & 0  \\
0 & 0 & 0 & 0 & 0 &  \cdots & 1& \lambda-2 & 1   \\
0 & 0 & 0 & 0 & 0 &  \cdots & 0 & 1 & \lambda-1  \\
\end{vmatrix}  =0.
$$
Expanded by the second row, we have
\begin{eqnarray*}
\det(\lambda I-L') & = & (\lambda-1)[(\lambda-t-1)D_{n-t-1}-D_{n-t-2}]-tD_{n-t-1}\\
& = & [\lambda^2-(t+2)\lambda+1]D_{n-t-1}-(\lambda-1)D_{n-t-2},
\end{eqnarray*}
where
\begin{eqnarray*}
D_{n-t-1} & = &
\begin{vmatrix}
\lambda-2 & 1 & 0 & \cdots & 0 & 0 & 0  \\
1 & \lambda-2 & 1 & \cdots & 0 & 0 & 0  \\
0 & 1 & \lambda-2 &  \cdots & 0 & 0 & 0  \\
\vdots & \vdots &\vdots&  \ddots & \vdots & \vdots & \vdots  \\
0 & 0 & 0 &  \cdots & \lambda-2 & 1 & 0  \\
0 & 0 & 0 &  \cdots & 1& \lambda-2 & 1   \\
0 & 0 & 0 &  \cdots & 0 & 1 & \lambda-1  \\
\end{vmatrix}\\
& = & \sum_{i=0}^{n-t-1}(-1)^{i}\binom{2n-2t-2-i}{i}\lambda^{n-t-1-i}.
\end{eqnarray*}
Thus the Laplacian characteristic polynomial of $B_{n}^{t}$ is
\begin{eqnarray*}
\Phi(B_{n}^{t}, \lambda) & = & (x-1)^{t-1}\det(\lambda I-L')\\
& = & (\lambda-1)^{t-1}(\lambda^2-(t+2)\lambda+1)\sum_{i=0}^{n-t-1}(-1)^{i}\binom{2n-2t-2-i}{i}\lambda^{n-t-1-i}\\
& & -(\lambda-1)^t\sum_{i=0}^{n-t-2}(-1)^{i}\binom{2n-2t-4-i}{i}\lambda^{n-t-2-i}.
\end{eqnarray*}

This completes the proof. $\Box$

\end{proof}

\begin{figure}[!hbpt]
\begin{center}
\includegraphics[scale=1.1]{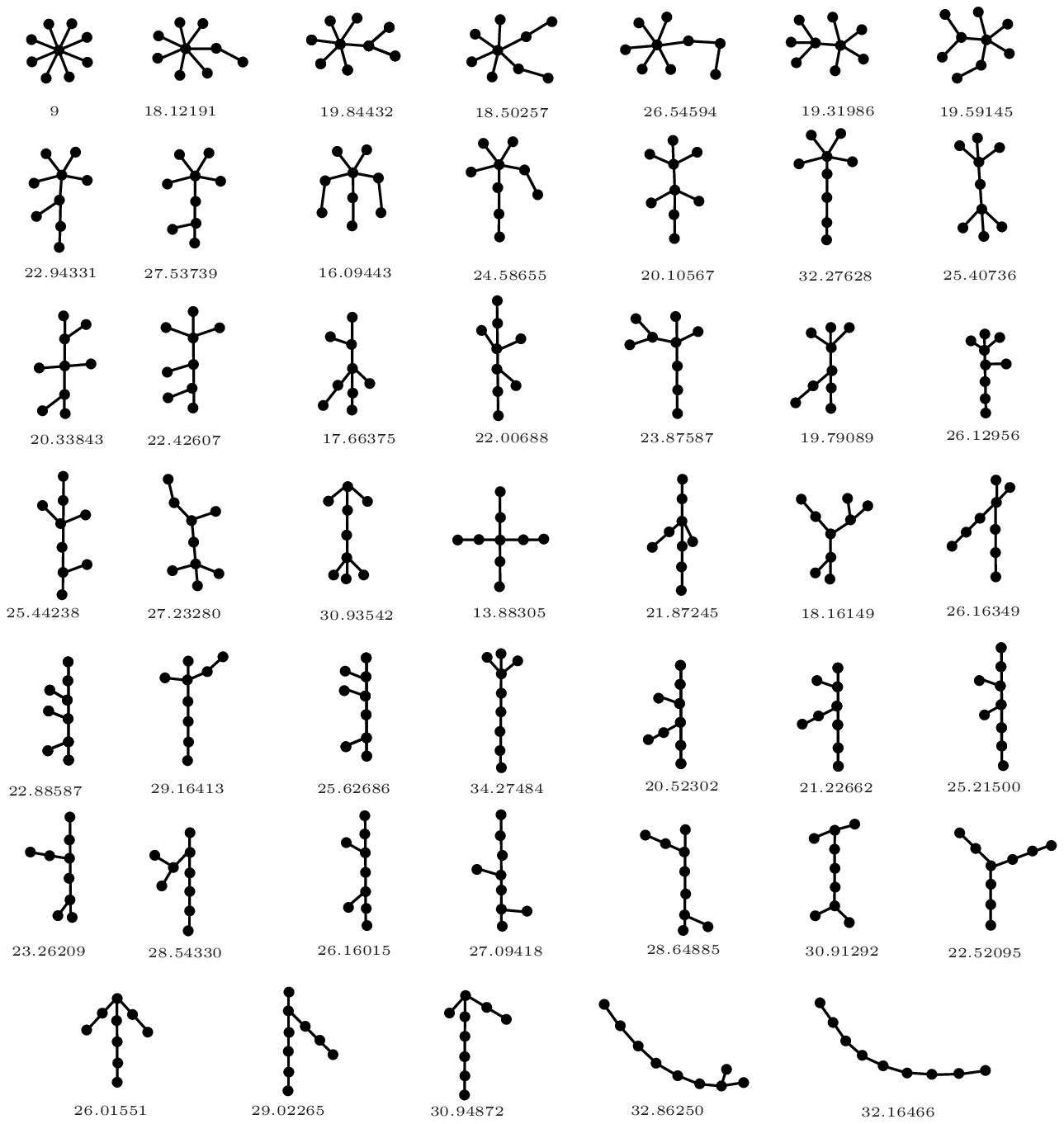}\\
Figure 3.~ The Laplacian spectral ratio of trees with nine vertices.
\end{center}\label{fig02}
\end{figure}

\begin{conjecture}\label{con5,1} 
Let $T$ be a tree with  $n\geq 8$ vertices. Then
$$R_L(K_{1,\,n-1})\leq R_L(T)\leq \begin{dcases}
R_L(B_{n}^{\frac{n-3}{2}}), & \text{if}\,\, n \,\,\text{is} \,\,\text{odd};\\
R_L(B_{n}^{\frac{n-4}{2}}), & \text{if}\,\, n \,\, \text{is} \,\, \text{even}.
\end{dcases}
$$
The left (right) equality holds if and only if $T= K_{1,\,n-1}$ ($T= B_{n}^{\frac{n-3}{2}}$ or $T= B_{n}^{\frac{n-4}{2}}$).
\end{conjecture}

\small {

}

\end{document}